\documentclass[11pt,reqno]{amsart}
\usepackage{amscd,amssymb,amsmath,amsthm}
\usepackage[english]{babel}
\usepackage[T1]{fontenc}
\usepackage[arrow,matrix]{xy}
\usepackage{graphicx}
\usepackage{tikz}
\usetikzlibrary{positioning,matrix,arrows,calc}
\usepackage{cite}
\usepackage{dsfont}
\usepackage{bbm}
\newtheorem{theorem}{Theorem}[section]
\newtheorem{lemma}[theorem]{Lemma}

\theoremstyle{definition}
\newtheorem{remark}[theorem]{Remark}

\numberwithin{equation}{section} 

\newcommand{\wC}{\widehat{\mathbb{C}}}

\begin{document}

\title[]{Critical intermittency in rational maps}

\author[{\tiny{Ale Jan Homburg}}]{Ale Jan Homburg}
\address{KdV Institute for Mathematics, University of Amsterdam, Science park 107, 1098 XH Amsterdam, Netherlands}
\address{Department of Mathematics, VU University Amsterdam, De Boelelaan 1081, 1081 HV Amsterdam, Netherlands}
\email{a.j.homburg@uva.nl}
\author[{\tiny{Han Peters}}]{Han Peters}
\address{KdV Institute for Mathematics, University of Amsterdam, Science park 107, 1098 XH Amsterdam, Netherlands}
\email{h.peters@uva.nl}
\author[{\tiny{Vahatra Rabodonandrianandraina}}]{Vahatra Rabodonandrianandraina}
\address{KdV Institute for Mathematics, University of Amsterdam, Science park 107, 1098 XH Amsterdam, Netherlands}
\email{v.f.rabodonandrianandraina@uva.nl}

\begin{abstract}
	Intermittent dynamics is characterized by
	long periods of different types of dynamical characteristics, for instance almost periodic dynamics alternated by chaotic dynamics.
	Critical intermittency is intermittent dynamics that can occur in iterated function systems, and involves a superattracting periodic orbit.
	
	This paper will provide and study examples of iterated function systems
	by two rational maps on the Riemann sphere that give rise to critical intermittency.
	The main ingredient for this is a superattracting fixed point for one map that is mapped onto a common repelling fixed point by the other map. We include a study
	of topological properties such as topological transitivity.
\end{abstract}

\maketitle

\section{Introduction}

This paper will provide and study examples of iterated function systems
by two rational maps on the Riemann sphere that give rise to intermittent
time series. The central example of the paper is intermittency of a type that we call critical intermittency, where the main ingredient is a superattracting fixed point for one map that is mapped by the other map onto a common repelling fixed point. We consider a topological
description of the dynamics for which we study density of orbits of the semi group generated by the iterated function system. And we consider
a metrical description by looking at statistical properties of the intermittent time series.

In dynamical systems theory, intermittency stands for time series that
alternate between different characteristics.
It in particular indicates times series that appear stationary
over long periods of time and are interrupted by bursts of nonstationary dynamics. These are called the laminar phase and relaminarization.
Explanations for the occurrence of intermittent time series were given by
Pomeau and Manneville \cite{pm80}, see also \cite{bpv86}.
They offered explanations using bifurcation theory,
and distinguished different types of intermittency
caused by different local bifurcations.
Later research added to the list of mechanisms
giving intermittency, including
crisis induced intermittency, homoclinic intermittency,
on-off intermittency and in-out intermittency.

%


\subsection{Critical intermittency in iterated function systems of logistic maps}

The type of dynamics we consider in this paper is related to the following example from interval dynamics.

Denote $g_a (x) = a x (1-x)$ for the logistic map on the interval $[0,1]$,
with $0<a \le 4$.
Consider the iterated function system generated by the two maps
$f_0 = g_2$ and $f_1 = g_4$.
This defines a semi-group $\langle f_0,f_1 \rangle$ 
of compositions
of $f_0$ and $f_1$.
For each iterate we pick $i\in\{0,1\}$ at random, i.i.d., with probabilities $p_0$ and $p_1 = 1-p_0$, and then iterate using $f_i$.
Note that this defines a Markov process and recall that a stationary measure for the Markov process is a measure $m$ satisfying
\[
m = p_0 f_0 m + p_1 f_1 m.
\]
Here $f_i m$, $i=0,1$, stands for the push forward $f_i m (A) = m (f_i^{-1} (A))$.
Carlsson \cite{c02} observed that the only stationary measure for this iterated function system is the delta measure at $0$ if $p_0 > 1/2$.
This was further investigated in \cite{agh18}
where $\sigma$-finite stationary measures were constructed for $p_0 > 1/2$,
and in \cite{hometal21} that studied stationary measures for all values of $p_0$.

We summarize results on this iterated function system
in the theorem below.
Write $\Sigma = \{0,1\}^\mathbb{N}$ and endow $\Sigma$ with the product topology and the Borel $\sigma$-algebra.
Further write $\omega = (\omega_i)_{i\in\mathbb{N}}$ for elements of $\Sigma$.
We denote
\[
f^n_\omega (x) = f_{\omega_{n-1}} \circ \cdots \circ f_{\omega_0} (x).
\]
On $\Sigma$ we consider the Bernoulli measure $\nu_{p_0,p_1}$ given the probabilities $p_0,p_1$.

\begin{theorem}[see \cite{agh18}]
Consider the iterated function system 	$\langle f_0,f_1 \rangle$ on
$[0,1]$ given with probabilities $p_0,p_1$.


Assume $p_1 \ge 1/2$.
Then the delta measure $\delta_0$ at $0$ is the unique stationary probability measure.
There is an absolutely continuous  $\sigma$-finite stationary measure 
with support equal to $[0,1]$.
For any $\varepsilon >0$ and for Lebesgue almost any $x \in [0,1]$,
\begin{enumerate}
	\item
	$f^n_\omega (x) \not \in (0,\varepsilon)$ for infinitely many $n$;
	\item
	$
	\lim_{N\to \infty} \frac{1}{N} \left\vert \{ 0 \le n < N \; : \; f^n_\omega (x) \in (0,\varepsilon) \} \right\vert = 1,
	$
\end{enumerate}
for almost all $\omega \in \Sigma$.
\end{theorem}	

The theorem expresses the occurrence of intermittent time series; orbits spend most time near $0$ but there are infrequent but repeated bursts away from $0$.
Under the conditions of the theorem,
the only stationary probability measure is $\delta_0$, even though
$0$ is repelling for both maps $f_0,f_1$.
The explanation lies in the existence of a superattracting fixed point $1/2$ for
$f_0$, which will be mapped onto the repelling fixed point $0$ after iterating under $f_1$. Iterates of $f_0$
bring  a point superexponentially close to $1/2$, after iterating $f_1$
the point will get superexponentially close to $0$, after which
many iterates are needed to escape neighborhoods of $0$.

The bound $p_1 \ge 1/2$ is optimal: it is shown in \cite{hometal21} that for $p_1 < 1/2$ there does exist an absolutely continuous
stationary probability measure.

\medskip

Other examples of this phenomenon, which we call critical intermittency,
are possible \cite{hometal21}. In section 2 we will introduce conditions on pairs of $1$-parameter families of rational functions that guarantee critical intermittency. Two key assumptions are the existence of a joint fixed point, which is attracting for the first map, and repelling for the second. In order for critical intermittency to occur, we assume the common fixed point is repelling on average. The first map will moreover have a superattracting fixed point, which is mapped back to the common fixed point by the second map. The proof of intermittency is aided by proving the density of semigroup orbits, which is shown for parameter values for which the common fixed point is not resonant.

In section 3 we introduce an explicit pair of $1$-parameter rational maps of degree $2$  that satisfy the conditions that imply critical intermittency for non-resonant parameter values. We moreover show that density of orbits still frequently occurs for parameter values where the common fixed point is resonant. In section 4 we will treat another explicit example, where the common fixed point is on average neutral instead of repelling. We show that intermittency can still occur, just as for real interval maps.

\section{Iterated function systems on the Riemann sphere}

In this section we will cover key assumptions that cause critical intermittency for rational iterated function systems on the Riemann sphere. We will work with a pair of $1$-parameter families of rational functions, and show that critical intermittency holds for almost every parameter value. In later sections we will discuss explicit pairs of rational functions where it is possible to decide more precise when critical intermittency holds.

\medskip

\noindent {\bf Key assumptions.} Throughout this section let $f_0 = f_{0,\lambda}$ and $f_1 = f_{1,\lambda}$ be rational functions, both depending analytically on a parameter $\lambda \in \mathbb D$. We will stipulate the behavior of $f_0$ and $f_1$ on two marked points in $\widehat{\mathbb C}$. Without loss of generality we may assume that these two points are $0$ and $\infty$.

We assume that $0$ is a fixed point for each $f_0$ and $f_1$. We assume it to be repelling for each $f_0$
and attracting for each $f_1$. The point $\infty$ is assumed to be a superattracting fixed point of $f_0$, and mapped to the fixed point $0$ by $f_1$. We will write $d\ge 2$ for the local degree of $f_0$ at infinity, so that $f_0$ is locally conjugate to $z \mapsto z^d$. Finally, we assume that each function $f_0$ is hyperbolic, and that all the other attracting fixed or periodic points of $f_0$ are mapped to 
the immediate basin of
$\infty$ for $f_0$ by some iterate $f_1^k$. Note that by hyperbolicity, the attracting periodic points vary holomorphically with $\lambda$, hence the iterate $k$ only depends on the choice of attracting periodic point, not on the value of $\lambda$.

For each fixed $\lambda$ we write $\langle f_0, f_1\rangle$ for the iterated function system generated by the functions $f_0$ and $f_1$. We iterate by picking the maps at random, i.i.d., with probabilities $p_0,p_1$ for $f_0,f_1$. We write $\mu(\lambda) = f_0' (0)$ and $\nu(\lambda) = f_1'(0)$, where $|\mu(\lambda)| > 1$ and $|\nu(\lambda)| < 1$ for each $\lambda$.
 The Lyapunov exponent at $0$ is the average
$p_0 \ln |f_0' (0)| + p_1 \ln |f_1'(0)|$ and we assume this to be positive:
\begin{align}\label{e:lyappos_general}
p_0 \ln |\mu(\lambda)| +  p_1 \ln |\nu(\lambda)| &> 0.
\end{align}
This makes the common fixed point $0$ repelling on average.

\subsection{Dense semigroup orbits.}

Recall that 
the action of the semi-group $G = G_\lambda = \langle f_0, f_1\rangle$ is said to be \emph{topologically transitive} if for every non-empty and open $U, V \subset \wC$ there exists $g \in G$ with $g(U) \cap V \neq \emptyset$.

Given an invariant set $S = f_0(S) \cup f_1(S)$,
the action of $G$ is said to be \emph{minimal} on $S$ if for all $z \in S$ the $G$-orbit of $z$ is dense in $S$.
We are interested in $G$-orbits that are dense in $\wC$. As this can not hold for all points, it can for instance not hold for  $z = 0,\infty$, the best that we can hope for is
that $G$-orbits of all but finitely many points
are dense.

We start with a result on topological transitivity.

\begin{lemma} \label{lemma:ttgeneral}
For each $\lambda \in \mathbb D$ the action of the semi-group $G$ is topologically transitive.
\end{lemma}

\begin{proof}
Consider an open set $U \subset \widehat{\mathbb C}$. By hyperbolicity of $f_0$ the set $U$ must intersect the attracting basin of some attracting fixed point or periodic cycle. Let $V \subset U$ be an open connected subset contained in this basin. Then for some large $n$ the set $f_0^n(V)$ is contained in a small neighborhood of an attracting periodic point $x$ of $f_0$. Let $k$ be such that $f_1^k (x)$ is in a small neighborhood of  $\infty$.
Then $f_1^k\circ f_0^n (V)$ is contained in a small neighborhood of the point $\infty$.

Recall that near $\infty$ the map $f_0$ is holomorphically conjugate to a map $z \mapsto z^d$, for $d\ge 2$. It follows that for large $\ell$ the set $f_0^\ell \circ f_1^k\circ f_0^n (U)$ contains an annulus around the point $\infty$. Moreover, by increasing $\ell$ if necessary we can guarantee that the modulus of this annulus is arbitrarily large.

It follows that  $f_1 \circ f_0^\ell \circ f_1^k\circ f_0^n (U)$ contains a small annulus of arbitrarily large modulus around the point $0$. Since the repelling fixed point $0$ is a non-isolated point of the Julia set of $f_0$, and since $f_0$ acts in local coordinates as multiplication by the multiplier at the fixed point, it follows that a small annulus around $0$ of sufficiently large modulus must contain a point on the Julia set of $f_0$, and thus also an open neighborhood of such a point. It follows that
$$
\bigcup_{m \in \mathbb N} f_0^m \circ f_1 \circ f_0^\ell \circ f_1^k\circ f_0^n (U) \supset \widehat{\mathbb C} \setminus \mathcal{E}_{f_0},
$$
where $\mathcal{E}_{f_0}$ is the exceptional set of $f_0$, which contains at most $2$ points. This completes the proof.
\end{proof}

\begin{remark}
One can verify that the exceptional set of $f_0$ is in fact empty. By compactness of $\widehat{\mathbb C}$ it follows that there exist integers $n,k,\ell$, and $m$ such that
$$
f_0^m \circ f_1 \circ f_0^\ell \circ f_1^k\circ f_0^n (U) = \widehat{\mathbb C}.
$$
\end{remark}


\begin{theorem}\label{t:densegeneral}
Let $\lambda$ be such that the set
$$
S = S_\lambda := \overline{\{\mu(\lambda)^n \cdot \nu(\lambda)^m \; : \; n,m \in \mathbb N\}}
$$
coincides with $\mathbb C$. Then there exists an $r>0$ such that for all $z \in B_r(0)\setminus\{0\}$ the $G_\lambda$-orbit of $z$ is dense in $\widehat{\mathbb C}$.
\end{theorem}	

\begin{remark}
For $\nu(\lambda)\in B(0,1)$ one of the following must be satisfied:
\begin{enumerate}
\item $S = \mathbb C$. \label{i:1}
\item $S$ is a finite union of real lines passing through the origin. \label{i:2}
\item $S$ is a discrete union of concentric circles. \label{i:3}
\item $S$ is discrete. \label{i:4}
\end{enumerate}
It is possible that case \eqref{i:2}, \eqref{i:3}, or \eqref{i:4} occurs persistently for all $\lambda$, in which case we say that $G$ is \emph{persistently resonant}. If $G_\lambda$ is not persistently resonant, the equality $S = \mathbb C$ will hold for almost every $\lambda \in \mathbb D$. Indeed, the parameters $\lambda$ for which cases  \eqref{i:2}, \eqref{i:3}, and \eqref{i:4} hold are given by countably many real analytic equations in $\lambda$. Each such equation either holds throughout, in which case $G_\lambda$ is persistently resonant, or is satisfied in a real analytic subvariety of real dimension $1$.
\end{remark}

\begin{proof}[Proof of Theorem \ref{t:densegeneral}]
We will consider semi-group orbits that remain in $B_r(0)$, and accumulate on a small annulus around the origin of arbitrarily large modulus. The argument that concludes the proof of \ref{lemma:ttgeneral} can then be used to determine density of the orbit in all of $\widehat{\mathbb C}$.

Since we will remain in $B_r(0)$, we may assume that $r>0$ is sufficiently small so that we can use linearizing coordinates for the map $f_0$, i.e. $f_0 (z) = \mu(\lambda) z$. Since the multiplier is preserved under conjugation we obtain $f_1(z) = \nu(\lambda) z + O(z^2)$. Let us denote the local linearization map of this function by $\varphi = \varphi_\lambda$, which is unique once we demand that $\varphi(z) = z + O(z^2)$. Thus $\varphi \circ f_1 \circ \varphi^{-1}: z \mapsto \nu(\lambda) z$ on $B_r(0)$, by shrinking $r$ if necessary.

Consider diverging sequences $(m_j)$ and $(n_j)$ for which $\mu(\lambda)^{m_j}\cdot \nu(\lambda)^{n_j}$ converges to $w \in B_r(0) \setminus \{0\}$. Then
$$
f_0^{m_j} \circ f_1^{n_j} = \mu^{m_j}  \varphi^{-1} (\nu^{n_j}  \varphi (z_0)) \rightarrow w \cdot \varphi (z_0).
$$
It follows that the semi-group orbit of $z$ accumulates on a small annulus around $0$ of arbitrarily large modulus, which completes the proof.
\end{proof}

\begin{remark}\label{r:densegeneral} Theorem \ref{t:densegeneral} implies the density of the $G$-orbit $G(z_0)$ of $z_0$ for an arbitrary initial value $z_0$ whenever some element in $G(z_0)$ lies in a small punctured neighborhood of the origin, which is of course a necessary condition. It is also clear that this condition is not always satisfied. For example, if $z$ equals one of the attracting periodic points of $f_0$ and is mapped exactly onto $\infty$ by $f_1$, then the orbit $G(z)$ is finite. In general there could be different subsets in the Julia set of $f_0$ that are invariant under both $f_0$ and $f_1$, which we exclude in the Lemma below. We will later discuss explicit examples of the functions $f_0$ and $f_1$ for which we can deduce that there are no non-trivial invariant subsets, and as a result we obtain density for all but finitely many initial values $z_0\in \widehat{\mathbb C}$. In those examples we can moreover deduce that density also occurs for most resonant parameters $\lambda$, where $S$ satisfies case (ii), (iii), or (iv). In the resonant setting more precise knowledge of the higher order terms is required to deduce density of local orbits near the origin.
\end{remark}

Theorem~\ref{t:densegeneral} and Remark~\ref{r:densegeneral} yield the following result.

\begin{lemma}\label{lemma:denseverygeneral}
	Suppose that the hypotheses of Theorem~\ref{t:densegeneral} are satisfied. Moreover, assume that none of the attracting fixed or periodic points of $f_0$ are mapped exactly to $\infty$ under iteration by $f_1$. In addition, assume that for any $z_0 \in \wC \setminus \{0,\infty\}$ there is $g \in G_\lambda$ so that $g(z_0)$ lies in the immediate basin of attraction  of $\infty$ for $f_0$. Then for $z \in \wC\setminus \{0,\infty\}$ the $G_\lambda$-orbit of $z$ is dense in $\widehat{\mathbb C}$.
\end{lemma}
	

\subsection{Critical Intermittency}


As in the real setting we write $\Sigma = \{0,1\}^\mathbb{N}$ and endow $\Sigma$ with the product topology and the Borel $\sigma$-algebra. We write $\omega = (\omega_i)_{i\in\mathbb{N}}$ for elements of $\Sigma$.
The iterated function system $\langle f_0, f_1\rangle$ defines a skew product system $F : \Sigma\times \wC \to
 \Sigma\times \wC$ given by
 \[
 F(\omega,z) = (\sigma \omega , f_{\omega_0} (z)).
 \]
Here $\sigma$ is the left shift operator.
Denote
\[
F^n (\omega,z) = (\sigma^n \omega, f^n_\omega (z)).
\]
We write $[i_0\ldots i_k]$ for the cylinder
$\{ \omega \in \Sigma\; : \; \omega_0=i_0,\ldots,\omega_k=i_k\}$.
On $\Sigma$ we consider the Bernoulli measure $\nu_{p_0,p_1}$ for the probabilities $p_0,p_1$, which is invariant under $\sigma$.
A stationary measure $m$ for the iterated function system
defines an invariant measure $\nu_{p_0,p_1}\times m$ for $F$.

\begin{theorem}\label{t:inter_general}
	Consider the iterated function system $\langle f_0,f_1\rangle$
	given with probabilities $p_0,p_1$. Suppose that the key assumptions and the additional hypotheses of Lemma~\ref{lemma:denseverygeneral} hold.
	Assume 
$$
p_0 > \frac{1}{d}.
$$
	Then the delta measure $\delta_0$ is the only finite stationary measure.
	Moreover, for any $\varepsilon >0$, and for Lebesgue almost any $z \in \wC$,
	\begin{enumerate}
		\item
		$f^n_\omega (z) \not \in B(0,\varepsilon)$ for infinitely many $n$;
		\item
		$
		\lim_{N\to \infty} \frac{1}{N} \left\vert \{ 0 \le n < N \; : \; f^n_\omega (z) \in B(0,\varepsilon) \} \right\vert = 1,
		$
	\end{enumerate}
	for almost all $\omega \in \Sigma$.
\end{theorem}

In the proof we use Ka\u{c} theorem.
We recall a version we use. 
Consider a measurable map $f: X \to X$  with a finite invariant measure $\mu$.
Let $E \subset X$ with $\mu(E) >0$.
Define the first return time $R : E \to \mathbb{N}\cup \{\infty\}$  by
\[
R(x) = \min \{ i >0 \; : \; f^i(x) \in E \}.
\]
We use the statement that the average first return time to $E$ is finite.

\begin{theorem}[Ka\u{c} theorem, see \cite{vo16}]
Let $f: X \to X$ be a measurable map with a finite invariant measure $\mu$. Let $E \subset X$ with $\mu(E) >0$.

Then
\[
\int_E  R(x) \, d\mu(x)  < \infty.
\]
If $\mu$ is an ergodic invariant measure, then $\int_E  R(x) \, d\mu(x)  = \mu (X)/\mu(E)$.
\end{theorem}

\begin{proof}[Proof of Theorem~\ref{t:inter_general}]
Assume there is a finite stationary measure $m$ that assigns
zero measure to $\{0,\infty\}$.
%
By Lemma~\ref{lemma:denseverygeneral} the support of $m$ is all of $\wC$.
Then $\nu_{p_0,p_1}\times m$ is a finite invariant measure for $F$.
(We may assume that $\nu_{p_0,p_1}\times m$ is ergodic.)
Given a set $A\subset \Sigma\times \wC$ of positive measure
$\nu_{p_0,p_1}\times m (A) >0$, Ka\u{c} theorem
yields finite average return time.
We will derive a contradiction
by providing a set $A$ of positive measure and with infinite average return time.

For $A$ we take the product set $A = [0] \times \mathcal{A}$
where $\mathcal{A}$ is an annulus around $0$ between
a small circle $C(0,\delta)$ around $0$ and $f_0\big(C(0,\delta)\big)$. Take $A$ so that it includes
$[0]\times C(0,\delta)$ but excludes $[0]\times f_0\big(C(0,\delta)\big)$. Since the support of $m$ is all of $\wC$ and $\nu_{p_0,p_1}\times m$ is a product measure, we find $\nu_{p_0,p_1}\times m(A)>0$.
Then Ka\u{c} theorem yields
\[
\int_A R(\omega,z)  \, d\nu_{p_0,p_1}\times m(\omega,z) <\infty
\]	
for the first return time $R$ to $A$.

Since $\infty$ is a superattracting fixed point for $f_0$ and $f_1(\infty) = 0$, it follows that for every $z\in \wC$ with $|z|\ge R$ we have $f_1\circ f_0^N(z)\in B(0,\delta)$ for large enough $R$ and $N$ larger than some $N_0$.

A calculation in the spirit of \cite{c02,agh18} establishes infinite average return time to $\mathcal{A}$ for $p_0 >1/d$. Recall that the local degree of $f_0$ at infinity is $d$. Thus for given $z\in\mathbb{C}$ with $|z|\ge R$ we have $|f_1\circ f_0^N(z)|\le  c/|z|^{d^N}$ for some $c>0$. Each iterate maps a point in $B(0,\delta)$ at most a constant factor further away from $0$. Therefore for a given $\omega\in\Sigma$, 
$f^M_\omega\circ f_1\circ f_0^N(z)$ may enter $\mathcal{A}$ only if $M>Cd^N$ for some $C>0$ and
\[
	\int_A R  \, d\nu_{p_0,p_1}\times m  \ge \sum_{i\ge N_0} \int_{[0^i1]\times\mathcal{A}} R  \, d\nu_{p_0,p_1}\times m \ge C \sum_{n \ge N_0} p_0^n d^n
\]
for some $C>0$. Hence $\int_A R  \, d\nu_{p_0,p_1}\times m=\infty$ if $p_0>1/d$.

Consequently the only stationary probability  measure is the delta measure $\delta_0$.	
Item~(1) follows from \eqref{e:lyappos_general}, 
which implies that for $z \in B(0,\varepsilon)$, with probability one the orbit $f^n_\omega(z)$ leaves $B(0,\varepsilon)$.	
	
We continue with item~(2).	We follow reasoning from \cite{as03} which is also used in \cite{agh18}.
Instead of using $B(0,\varepsilon)$ we find it convenient to replace it with the union $W$ of $B(0,\varepsilon)$ and a small disc $\{ z \in \mathbb C \; : \; |z|>r \} \cup \{\infty\}$ in $\wC$ around $\infty$.
We will establish that for almost all $\omega$,  $f^n_\omega (z)$ spends on average a bounded number of iterates between leaving and re-entering $W$.
Item~(2) in the formulation of the theorem will be deduced from this, and
it gives in fact additional information on the duration of the relaminarization.


Given $\varepsilon>0$, there is a finite partition
$\{ Q_i\}$
of $\wC \setminus W$ so that for $Q_i$
there is a cylinder $B_i$ of uniformly bounded depth $K_i \le  K$,
so that $f_\omega^{K_i}(z) \in W$ for
$(\omega,z)\in B_i\times Q_i$.
As $\nu_{p_0,p_1}(B_i)$ is bounded from below, it follows that the expected first hitting time for
$z \in \wC \setminus W$ to enter $W$
is finite: if $U(\omega,z) = \min\{i >0 \; : \; f_\omega^i(z) \in W \}$, then
$U(\omega,z) < \infty$ almost surely and
\[
\int_{\Sigma} U(\omega,z) \, d \nu_{p_0,p_1} < \infty
\]
uniformly in $z$.

Take an orbit $z_n = f^n_\omega (z_0)$ with for definiteness $z_0 \in W$. Also assume that $z_0$ is not in the inverse orbit of $0$.
Define subsequent escape times from $W$ and $\wC \setminus W$: $T_0 = 0$ and
\begin{align*}
T_{2k+1} &= \inf \{ n \in \mathbb{N} \;  \mid  \; n > T_{2k} \text{ and } z_n \not \in  W \},
\\
T_{2k} &= \inf \{ n \in \mathbb{N} \; \mid \; n > T_{2k-1} \text{ and } z_n \in W \}.
\end{align*}
Note that such a sequence of escape times exists almost surely.
Write $\eta_k = T_{2k-1} - T_{2k-2}$ and $\xi_k = T_{2k} - T_{2k-1}$
for the duration of the orbit pieces in $W$ and $\wC\setminus W$.
Define for $n \in [T_{2k},T_{2k+1})$,
\begin{align*}
N_\eta (n) = k, &\;
N_\xi (n) = k
\end{align*}
and
$\tilde{\eta} (n) =  n +1 - T_{2k}$, $\tilde{\xi} (n) = 0$,
so that $\tilde{\eta}$ counts the number of iterates from $T_{2k}$ on where $z_n \in W$.
Likewise define for $n \in [T_{2k+1},T_{2k+2})$,
\begin{align*}
N_\eta (n) = k+1, &\;
N_\xi (n) = k
\end{align*}
and
$\tilde{\eta}(n) = 0$,
$\tilde{\xi} (n) =  n +1 - T_{2k+1}$,
so that $\tilde{\xi}$ counts the number of iterates from $T_{2k+1}$ on where $z_n \not\in W$.

Finally calculate
\begin{align}
\nonumber
& \frac{1}{n} \sum_{i=0}^{n-1}   \mathbbm{1}_{W} (f^i_\omega (x_0)) =
\frac{1}{n} \left( \sum_{k=1}^{N_\eta (n-1)} \eta_k + \tilde{\eta} (n-1) \right)
\\
\nonumber
&= \left( \sum_{k=1}^{N_\eta(n-1)} \eta_k + \tilde{\eta}(n-1) \right) \left/ \left( \sum_{k=1}^{N_\eta(n-1)} \eta_k + \tilde{\eta}(n-1) + \sum_{k=1}^{N_\xi (n-1)} \xi_k + \tilde{\xi} (n-1) \right) \right.
\\
\nonumber
&= \left(  1 + \left( \sum_{k=1}^{N_\xi(n-1)} \xi_k + \tilde{\xi} (n-1)   \right) \left/ \left( \sum_{k=1}^{N_\eta(n-1)} \eta_k + \tilde{\eta}(n-1) \right)\right.    \right)^{-1}
\\
\label{e:frequency}
&\ge \left( 1 + \left( \sum_{k=1}^{N_\xi (n-1) + 1} \xi_k \right) \left/ \left( \sum_{k=1}^{N_\eta(n-1)} \eta_k \right)\right.\right)^{-1}.
\end{align}

Construct independent stochastic variables
$\sigma_k\ge \xi_k$ that have uniformly bounded expectation and variance.
This can be done as follows:
by shrinking $B_i$ one constructs cylinders $B_i$ of constant depth $K$ and
with constant measure $\nu_{p_0,p_1}(B_i)$, still satisfying $f_\omega^K (Q_i) \subset W$ for $\omega\in B_i$.
Define $G: \Sigma \times \wC \to \Sigma \times \wC$ by
$G \equiv F$ on $\Sigma \times W$ and $G (\omega,z) = (\sigma \omega , 0)$ for $z\in W$. Take a cylinder $B$ in $\Sigma$ of depth $K$ and measure
$\nu_{p_0,p_1} (B) = \nu_{p_0,p_1}(B_i)$ and add $(B,W)$ to the collection $\{(B_i,Q_i)\}$.
Consider the stochastic variable $\sigma$ that gives the first time to enter a $B_i\times Q_i$ by iterating $G^K$.
Take independent copies $\sigma_k \ge \xi_k$ of $\sigma$.

An application of the strong law of large numbers (see for instance \cite{ks07})  gives that
$\frac 1 m \sum_{k=1}^{m} \xi_k \le  \frac 1 m \sum_{k=1}^m \sigma_k$ stays bounded for large $m$, almost surely.
For $z \in W$, let $V(\omega,z) = \min\{i >0\; : \; F^i (\omega,z) \in \wC \setminus W \}$.
Let $\rho$  be the minimal escape time to  $\wC \setminus W$, minimized
over initial points $z \in W$; $\rho =\min_{z\in W} V(\omega,z)$.
There are independent copies $\rho_n$ of $\rho$ with $\rho_n \le  \eta_n$.
Now $\lim_{m\to\infty} \frac 1 m \sum_{i=0}^{m-1} \rho_i = \infty$
almost surely and hence
$\lim_{m\to\infty} \frac 1 m \sum_{i=0}^{m-1} \eta_i = \infty$
almost surely.
We conclude that the last term in \eqref{e:frequency} goes to $1$ for almost all $\omega$, as $n \to \infty$ (note that $N_\eta (n-1) - N_\xi(n-1) \le  1$). Item~(2)
follows since the average escape time out of $\{ z \in \mathbb C \; : \; |z|>r\} \cup \{ \infty\}$ is finite.
\end{proof}

\begin{remark}
	We get that for any $z\in \wC$,
	\[
	\lim_{n\to\infty} \frac{1}{n} \sum_{i=0}^{n-1} \delta_{f^i_{\omega} (z)} = \delta_0,
	\]
	where the convergence is in the weak star topology, for $\nu_{p_0,p_1}$ almost all $\omega$.
\end{remark}

\section{Degree $2$ example}
	
In this section we discuss an explicit pair of $1$-parameter rational functions for which critical intermittency can be shown. Given a parameter $\lambda \in \mathbb C$, consider the two maps
\begin{align}\nonumber
f_0 (z) &= 2z + z^2,  
\\
f_1 (z) &= \lambda \frac{z}{(z+1)^2} \label{e:f0f1}
\end{align}
on the Riemann sphere $\wC$.
The first map, $f_0$, is conjugate to $z \mapsto z^2$ through a translation by $1$. Its Julia set equals the circle $\{ |z+1| =1 \}$.	
The maps $f_0$ and $f_1$ have $0$ as common fixed point.	
Moreover, the set of three points $\{ 0,\infty, -1 \}$
is forward invariant by both maps. The points are mapped under
$f_0$ , $f_1$ in the following way:
\begin{align*}
f_0 (0) &= 0, & f_1(0) &= 0,
\\
f_0 (\infty) &= \infty, & f_1(\infty) &= 0,
\\
f_0 (-1) &= -1, & f_1(-1) &= \infty.
\end{align*}

As in the general setting we write
 $\langle f_0,f_1\rangle$ for the iterated function system generated
 by $f_0,f_1$, and work with associated probabilities $p_0, p_1$. Again we assume that the Lyapunov exponent at $0$ is positive:
\begin{align}\label{e:lyappos}
 p_0 \ln |2| +  p_1 \ln |\lambda| &> 0.
\end{align}
We will use the skew product notation
\[
F(\omega,z) = (\sigma \omega , f_{\omega_0} (z))
\]
introduced in the previous section.

\subsection{Dense semi-group orbits}	

Since our explicit semigroup satisfies the general assumptions from the previous section, we immediately obtain topological transitivity from Lemma \ref{lemma:ttgeneral}. Our next goal is to prove the density of orbits. Lemma \ref{lemma:denseverygeneral} implies that density occurs for almost all values of $\lambda$, and for initial values sufficiently close to $0$. We will see that in our explicit setting density also occurs frequently for parameter values for which the semigroup $G_\lambda$ is resonant.

Of course, for $\lambda \in \mathbb R$ the invariance of $\mathbb R$ implies that density does not occur for real initial values. We will focus on proving density when $\lambda \in B(0,1) \setminus \mathbb R$. We need a lemma on linearizing coordinates.

\begin{lemma}\label{l:linear}
	Let $\lambda \neq 0$. Then the rational functions $f_0$ and $f_1$ are not simultaneously linearizable at $0$.
\end{lemma}
\begin{proof}
	Recall that the linearization map is unique up to a multiplicative constant. It is therefore sufficient to consider the linearization map $\varphi$ for $f_0$ of the form $\varphi(z) = z + a_2 z^2 + a_3 z^3 + O(z^4)$, and to show that $\varphi$ does not also linearize $f_1$. Since $f_0(z) = 2z + z^2$ we have that $a_2 = -1/2$ and $a_3 = 1/6$. Observe that
	$$
	f_1(z) = \frac{\lambda z}{1 + 2z + z^2} = \lambda z - 2 \lambda z^2 + 3\lambda z^3 + O(z^3).
	$$
	It follows that
	$$
	\varphi \circ f_1 - \lambda \varphi = \frac{1}{6}\lambda z^2\left(-3(3+\lambda) + (17 + 12 \lambda + \lambda^2)z + O(z^2)\right).
	$$
	Therefore the second order part of $\varphi \circ f_1 - \lambda \varphi$ only vanishes when $\lambda = -3$, in which case
	$$
	\varphi \circ f_1 - \lambda \varphi = 5z^3 + O(z^4).
	$$
	Hence regardless of the value of $\lambda \neq 0$ the maps $f_0$ and $f_1$ cannot be simultaneously linearizable.
\end{proof}

Next we exclude nontrivial forward invariant sets for the semi-group outside $\{-1,0,\infty\}$.

\begin{lemma}\label{l:noinv}
	Let $\lambda = \mathbb C \setminus \mathbb R$. If the semi-group orbit of $z \in \wC$ is contained in $\{|z+1| = 1\} \cup \{-1, \infty\}$, then $z \in \{0, -1, \infty\}$.
\end{lemma}
\begin{proof}
	For simplicity of notation we work with $w = z+1$, which gives $f_0: w \mapsto w^2$ and
	$$
	f_1:w \mapsto \frac{\lambda(w-1)}{w^2} +1.
	$$
	
	Assume that both $w$ and $f_1(w)$ are contained in $\partial \mathbb D$.
	Analysis of the image of $\partial \mathbb D$ under $w \mapsto \frac{w-1}{w^2}$ shows that for each fixed $\lambda$ there are at most $4$ points $w \in \partial \mathbb D$ whose image can lie in $\partial \mathbb D$, including the fixed point $w =1$.
	\begin{figure}[!ht]
		\begin{center}
			\includegraphics[height=4cm]{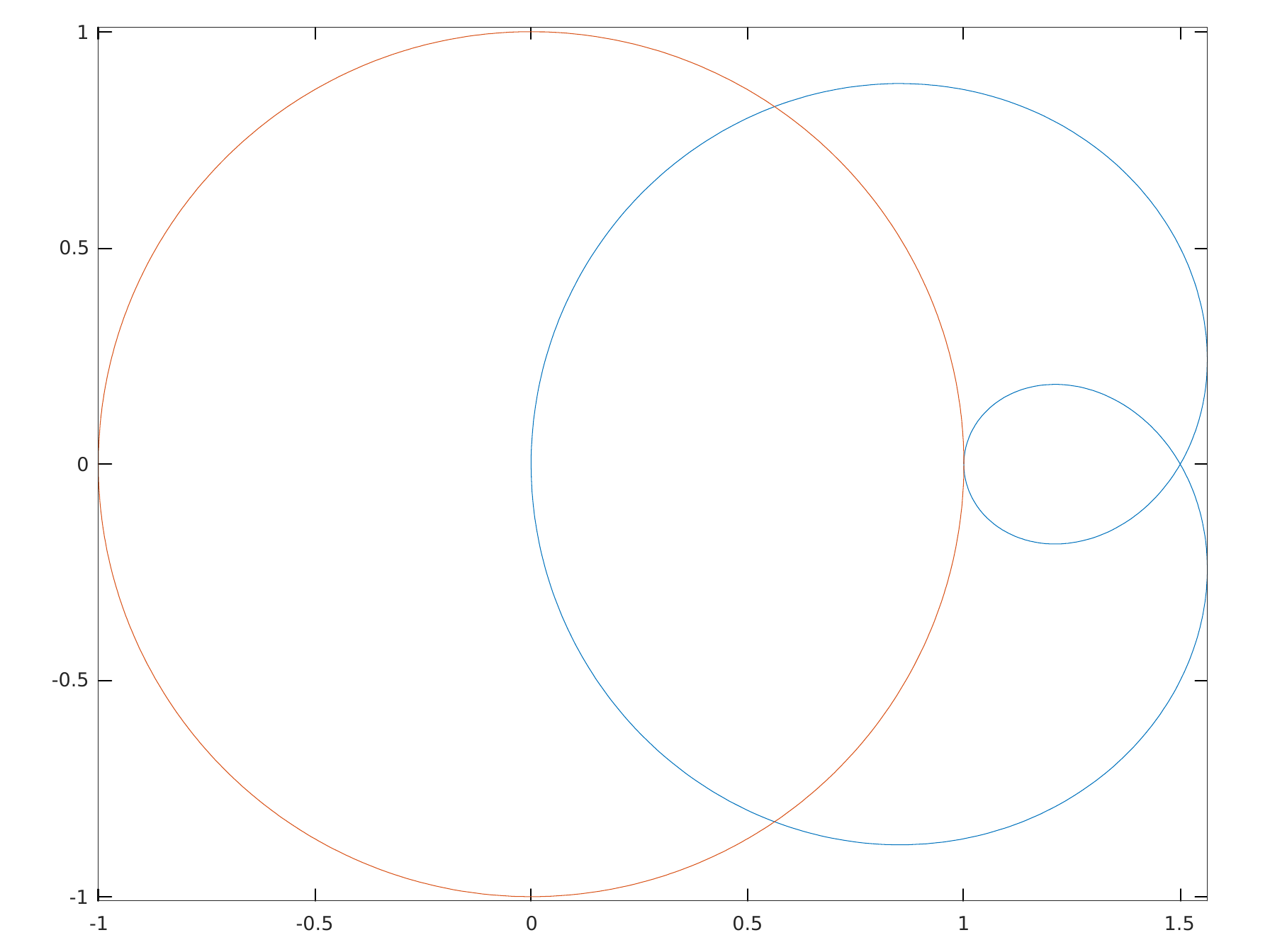}
			\includegraphics[height=4cm]{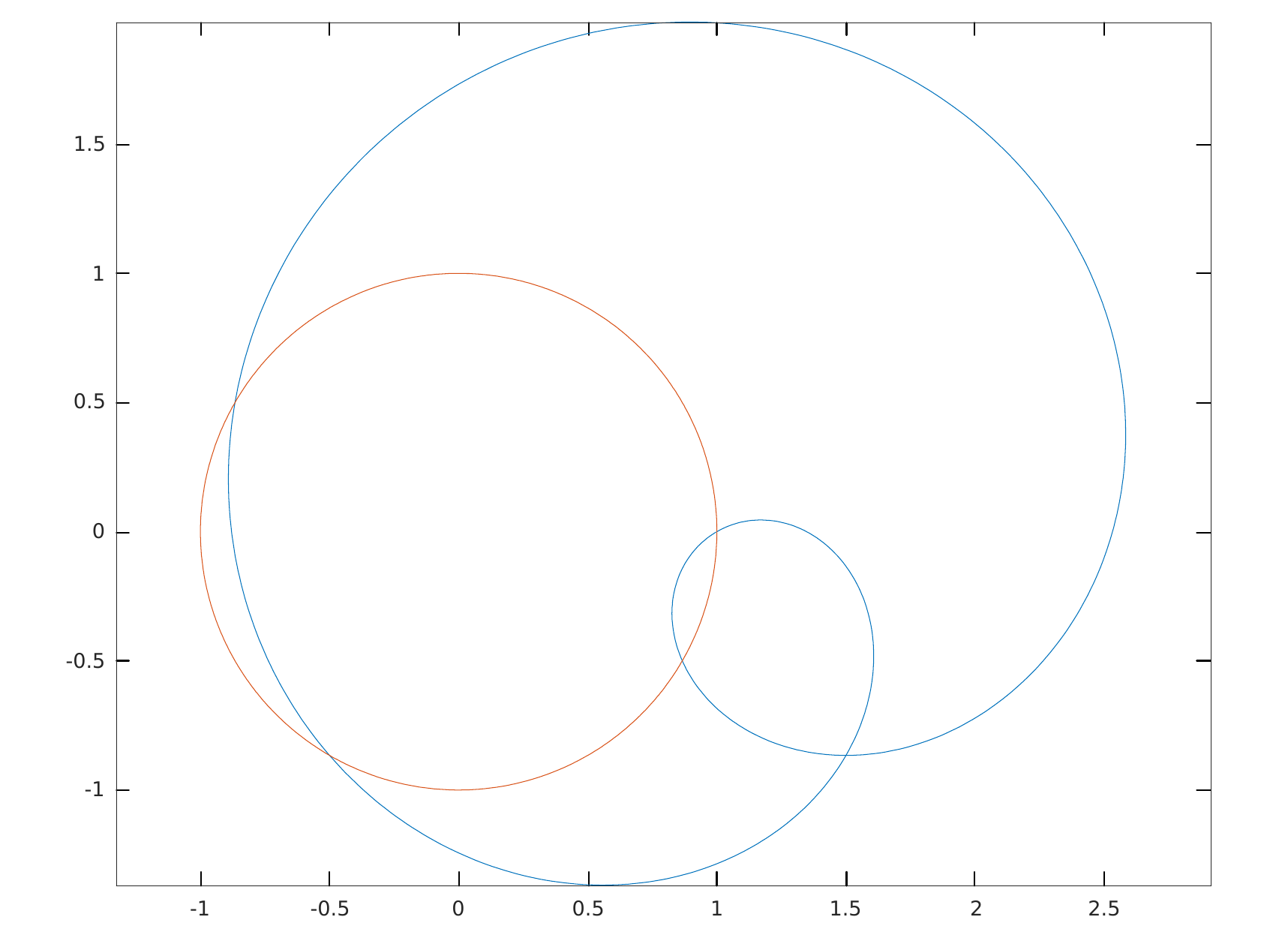}
			\caption{\label{f:curve} Left panel: The curve $f_1(\partial \mathbb D)$ for $\lambda = 1/2$, together with $\partial \mathbb D$.
				Right panel: $f_1(\partial \mathbb D)$ for $\lambda = \frac 1 2 + \frac 1 2 \sqrt{3}$.
			}
		\end{center}	
	\end{figure}
	Therefore it is sufficient to only consider $w \in \partial \mathbb D$ whose forward $f_0$-orbit contains at most $3$ points, possibly with $1$ added. Up to complex conjugation we therefore obtain the following eight candidates for a forward invariant sets.
	\begin{align*}
	&\{1, -1\},
	&\{1, i,-1 \},\\
	&\{1, e^{2\pi i/3}, e^{4\pi i/3} \},
	&\{1, -1, -i, i \}\\
	&\{1, e^{\pi i/4} ,i ,-1\}, &\{1, i, -1, e^{5\pi i/4} \}, \\
	&\{1, e^{2\pi i/7}, e^{4\pi/7}, e^{8\pi i/7}   \},
	&\{1, e^{\pi i/3}, e^{2\pi i/3}, e^{4\pi i/3}\}.
	\end{align*}
	A case by case analysis shows that for any of these sets the $f_1$-image does not leave the set invariant for $\lambda \notin \mathbb R$.
\end{proof}

\begin{theorem}\label{t:minimal}
	Let $\lambda \in B(0,1) \setminus \mathbb R$ and let $z_0 \in \mathbb C \setminus \{-1,0,\infty\}$. Then the $G$-orbit of $z_0$ is dense in $\wC$.
\end{theorem}

\begin{proof}
	We use the notation introduced in the proof of the previous proposition. Since $z_0 \notin \{-1, 0, \infty\}$ Lemma~\ref{l:noinv} implies that its $G$-orbit intersects $B(0,r) \setminus \{0\}$, hence we may assume that $z_0 \in B(0,r) \setminus \{0\}$. By topological transitivity of the $G$-action it is sufficient to prove that the $G$-orbit is dense in some open $U \subset B(0,r)$. We will prove this using only the local orbits $f_\omega^n(z_0)$ with $\omega \in \Omega$
	(recall that $\Omega$ is the set of all sequences $\omega \in \{0,1\}$ for which $f_\omega^n(z_0) \in B(0,r)$).
	
Define
$$
S:= \overline{\{2^m \cdot \lambda^n\}}.
$$
As in the general setting we distinguish four cases:
\begin{enumerate}
		\item[(i)] $S$ is dense in $\mathbb C$.
		\item[(ii)] $S$ is a union of real lines passing through the origin.
		\item[(iii)] $S$ is a union of concentric circles.
		\item[(iv)] $S$ is discrete.
\end{enumerate}
	
    The simplest case~(i) was treated in the general setting, in Lemma \ref{lemma:ttgeneral}. It is clear that in the other cases it is not sufficient to consider the linear part only, and one needs to take into account the higher order terms. An idea that will be used in several different places is the following: if the $G$-orbit of $a$ accumulates on a point $b$, and the $G$-orbit of $b$ accumulates on $c$, then the $G$-orbit of $a$ accumulates on $c$ as well.
	
	\medskip
	
	\noindent {\bf Case~(ii).} Observe that there exists a minimal $k$ such that $\lambda^k >0$. Use linearization coordinates as in case~(i).
 	  Since by Lemma~\ref{l:linear} $f_0$ and $f_1$ are not simultaneously linearizable at $0$  it follows that $\varphi$ is not linear.
	
	Denote by $H$ a union of $k$ radial half-lines, invariant under the linear map $z \mapsto \lambda z$. Then it follows that $g_1$ maps $\varphi^{-1} (H)$ into itself. Observe that $\varphi^{-1}(H)$ consists of $k$ real analytic curves, each tangent at the origin to a half-line.
	It follows that the $g_1$-orbit of $z_0$ is contained in $\varphi^{-1} (H_0)$ for some choice of half-lines $H_0$.
	As in case~(i), $g_0^m \circ g_1^n (z_0)$ converges to $2^m \lambda^n \varphi (z_0)$ as $2^m \lambda^n$ converges along a sequence with $m,n \to \infty$.
	It follows that the set of accumulation points of $g_0^{m} \circ g_1^n(z_0)$ contains a set of half-lines containing $\varphi(z_0)$.
	
	Since this description of the accumulation points holds for any base point $z_0$ in $B(0,r)$, we can apply it also to each of the points $z$ in one of the $k$ radial half-lines in $H_0$. Thus we obtain a union of $k$-half lines $H(z)$ that varies continuously with $z$. It remains to be shown that these $k$-half lines vary with $z$, and is not constant. But this follows since $\varphi(H)$ is \emph{not} a union of half-lines, since $\varphi$ is not linear. It follows that we obtain an open set of accumulation points, which completes the proof.
	
	\medskip
	
	\noindent {\bf Case~(iii).} We have that $2^m\cdot \lambda^n = e^{2\pi i \theta}$ for some $\theta \in \mathbb R \setminus \mathbb Q$. It follows that
	$$
	f_1^m \circ f_0^n(z) = e^{2\pi i \theta} z + h.o.t..
	$$
	Note that the rationality of $f_0$ and $f_1$ implies that the higher order terms are non-zero. Since $f_0$ and $f_1$ are not simultaneously linearizable, it follows that
	$$
	f_0^n \circ f_1^m \neq f_1^m \circ f_0^n,
	$$
	and hence
	$$
	f_1^{2m} \circ f_0^{2n} \neq \left(f_1^m \circ f_0^n\right)^2.
	$$
	Therefore the semi-group $H = \langle h_0, h_1\rangle$ induced by the two distinct neutral maps $h_0 := f_1^{2m} \circ f_0^{2n}$ and $h_1 := \left(f_1^m \circ f_0^n\right)^2$ with the same rotation number, cannot be normal in a neighborhood of $0$, see \cite{gp18}. In particular the action of $H$ on any closed curve winding around $0$ is unbounded.
	
	Observe that $\{f_0^m \circ f_1^n(z_0)\}_{m,n \in \mathbb N}$ accumulates on closed curves winding around $0$ of arbitrarily small radius. Hence there are points on those curves whose $H$-orbits must be unbounded. But since the generators of $H$ are both neutral at the origin, it follows that the unbounded orbit under $H$ starting at such a point must be arbitrarily dense on an annulus enclosed by two of such closed curves winding around $0$. It follows that by considering the $H$-action on smaller and smaller scales, and repeatedly composing with $f_0$ to get back to a given scale, we obtain density on an open set which completes the proof of case~(iii).
	
	\medskip
	
	\noindent {\bf Case~(iv).} Let us again work in linearization coordinates for $f_0$, such that we can write $g_0(z) = 2z$. By the assumption in case~(iv) it follow that $2^m \lambda^n = 1$ for certain minimal $n$ and $m$, from which it follows that
	$$
	h(z) := g_0^m \circ g_1^n(z) = z + h.o.t..
	$$
	Straightforward computation shows that the second order term of $g_1$ does not vanish when $|\lambda| < 1$ (Lemma~\ref{l:linear}),
	from which it follows that the second order term of $h$ also does not vanish. Thus $h$ has a parabolic fixed point at the origin with a single parabolic basin. In order to simplify the discussion we can apply a linear coordinate transformation to give $h(z) = z + z^2 + O(z^3)$, so that all orbits in the parabolic basin of $h$ converge to zero along the negative real axis.
	
	Consider base points $z_n = g_1^n(z_0)$ for $n$ large with $\mathrm{arg}(z_n)$ bounded away from $0$. Then $z_n$ lies in the parabolic basin of $h$ at $0$. Write $z_{n,j} = h^j(z_n)$ where $j$ runs from $0$ to a large $k = k(n)$ satisfying $|z_{n,k}| << |z_n|$. Recall that the points $z_{n,j}$ converge to zero as $j \rightarrow \infty$, along a real analytic curve tangent to the negative real axis. Moreover, the ratios between consecutive points satisfy
	$$
	\frac{z_{n,j}}{z_{n,j+1}} \rightarrow 1
	$$
	as $n \rightarrow \infty$. Write $w_{n,j} = g_1(z_{n,j})$, so that the points $w_{n,j}$ converge to zero along the half line through $-\lambda$, which is different from the negative real axis, and since $\lambda \notin \mathbb R$ also different from the positive real axis.  Choose $J>0$ such that $\mathrm{arg} (w_{n,j}) \sim \mathrm{arg}(-\lambda)$ for $j \ge J$. It follows that the points $w_{n,j}$ still lie in the parabolic basin for $j \ge J$. Now define $w_{n,j, \ell} = h^\ell(w_{n,j})$ for $j \ge J$ and $\ell \ge 0$.
	
	Recall the existence of the Fatou coordinate on the parabolic basin: a change of coordinates, again denoted by $\varphi$, conjugating $h$ to $z \mapsto z+1$. Recall that $\varphi$ is of the form $z \mapsto -\frac{1}{z} + b \log(z) + o(1)$ for some $b \in \mathbb C$. It follows that each of the orbits $\{w_{n,j, \ell}\}_{\ell \in \mathbb N}$ lies on a real analytic curve, and these real analytic curves are all transverse to the half line through $-\lambda$, with angles bounded away from zero. After scaling by an iterate $g_0^s$ to bring $w_{j,J,\ell}$ back to fixed scale, we obtain an arbitrarily dense set of points lying in an open set of uniform size. By increasing $n$ and taking a convergent subsequence of $g_0^s (w_{j,J,\ell})$ we obtain a dense set of accumulation points in an open subset, completing the proof.
\end{proof}

\begin{remark}
	It is not clear to the authors whether Theorem \ref{t:minimal} also holds
	for nonreal $\lambda$ with $|\lambda| > 1$. However, it does hold for generic $\lambda$.
	Assume that the set
	$$
	S^\prime := \overline{\{2^{-m} \cdot \lambda^n\}}
	$$
	is dense in $\mathbb C$. Using the attraction under $f_0$ to the point $\infty$, we can consider a starting value $z_0$ that is unequal to but arbitrarily close to $0$. We may therefore assume that, for $k \in \mathbb N$ large, the point $2^k z_0$ is still close to zero. For $j \le k$ we obtain that
	$$
	f_0^j f_1^n (z_0) \sim 2^j \lambda^n z_0 = 2^{-(k-j)} \lambda^n (2^m z_0)
	$$
	when $f_0^j f_1^n (z_0)$ is still sufficiently close to the origin. The assumption that $S^\prime$ is dense implies that by starting with smaller and smaller values of $z_0$, the set of points $f_0^j f_1^n(z_0)$ becomes more and more dense in a round annulus centered at $0$ of arbitrarily large modulus. As in the proof of Theorem~\ref{t:densegeneral} it follows that the semi-group orbit is dense in $\wC$.
\end{remark}

\begin{remark}
The Fatou set $F(G)$ of the semi-group $G = \langle f_0,f_1\rangle$ is the set of points where $G$ is normal. The Julia set $J(G)$, defined as the complement $\wC \setminus F(G)$,  is a closed backward invariant set.
Under the assumptions of Theorem~\ref{t:minimal},  $J(G)$ equals
$\wC$.
Indeed, the Julia set contains $0$ and by backward invariance also $\infty$. By \cite{hm96a} it contains a neighborhood of $\infty$ and then using Theorem~\ref{t:minimal} it equals $\wC$.
By \cite{hm96b}, repelling fixed points of elements of  $G$ lie dense in $J(G) = \wC$.
\end{remark}

\subsection{Intermittency}
Our explicit semigroup satisfies the hypotheses of Theorem~\ref{t:densegeneral}. Moreover Lemma~\ref{l:noinv} implies that the only invariant subset is $\{-1,0,\infty\}$. Thus the result of Theorem~\ref{t:inter_general} still holds for our specific example. As the density occurs for parameter values for which the semigroup is resonant, we can state the intermittency with larger values of $\lambda$.
\begin{theorem}\label{t:inter}
	Consider the iterated function system $\langle f_0,f_1\rangle$
	given with probabilities $p_0,p_1$. Let $\lambda \in B(0,1) \setminus \mathbb R$.
	Assume \eqref{e:lyappos} holds and  \[ p_0 > \frac{1}{2}.\]
	
	Then the only finite stationary measure is the delta measure $\delta_0$.
	For any $\varepsilon >0$, for Lebesgue almost any $z \in \wC$,
	\begin{enumerate}
		\item
		$f^n_\omega (z) \not \in B(0,\varepsilon)$ for infinitely many $n$;
		\item
		$
		\lim_{N\to \infty} \frac{1}{N} \left\vert \{ 0 \le n < N \; : \; f^n_\omega (z) \in B(0,\varepsilon) \} \right\vert = 1,
		$
	\end{enumerate}
	for almost all $\omega \in \Sigma$.
\end{theorem}

	\section{Vanishing Lyapunov exponents}
	
	For iterated function systems of interval maps,
	Gharaei and the first author~\cite{gh17} showed how intermittent time series
	occur if there is a common fixed point with a vanishing Lyapunov exponent,
	so that the common fixed point is neutral on average.
We will present an analogous example for iterated function systems
of M\"obius transformations on the Riemann sphere.
Consider the maps
\begin{align*}
f_0 (z) &= \mu z,
\\
f_1 (z) &= \frac{z}{\mu + z}
\end{align*}
We pick the maps $f_0,f_1$ with equal probability.

\begin{theorem}
	Consider the iterated function system $G = \langle f_0,f_1\rangle$
	given with probabilities $p_0=p_1=1/2$. Assume that $|\mu| >1$ and $\mu \not\in \mathbb R$.
	
	The $G$-orbit of any
	$z_0\in \wC \setminus \{0\}$ is dense.
	
	The only finite stationary measure is the delta measure $\delta_0$.
	For any $\varepsilon >0$, for Lebesgue almost any $z \in \wC$,
	\begin{enumerate}
		\item
		$f^n_\omega (x) \not \in B(0,\varepsilon)$ for infinitely many $n$;
		\item
		$
		\lim_{N\to \infty} \frac{1}{N} \left\vert \{ 0 \le n < N \; : \; f^n_\omega (x) \in B(0,\varepsilon) \} \right\vert = 1,
		$
	\end{enumerate}
	for $\nu_{p_0,p_1}$-almost all $\omega \in \Sigma$.
\end{theorem}

\begin{proof}
	That semi-group orbits lie dense  is proved as in Theorem~\ref{t:minimal}.
	To prove the remaining statements on intermittency we follow the reasoning of Theorem~\ref{t:inter}.
	Key statement is again
	\begin{equation}\label{eq:again}
	\int_A R  \, d\nu_{p_0,p_1}\times m = \infty,
	\end{equation}
	where
	\[
	R(\omega,z) = \min \{ i >0 \; : \; F^i(\omega,z) \in A \}
	\]
	is the return time to $A$.
	Here as before $A = [0] \times \mathcal{A}$ with
	$\mathcal{A}$ an annulus between a small circle $S$ around $0$ and its image $f_0(S)$.
	By Ka\u{c} theorem this implies there is no finite stationary measure with support intersecting $\mathcal{A}$, so that the only stationary probability measure is $\delta_0$.
	
	To obtain \eqref{eq:again}, fix $z$ inside $S$ and let $H(\omega) = \min \{ i >0 \; : \; F^i(\omega,z) \in A \}$ be the first time that $f^i_\omega(z)$ that enters $A$.
	We are done if we prove
	\begin{align}\label{e:Einfty}
	\int_\Sigma H \, d\nu_{p_0,p_1} &= \infty.
	\end{align}
	A sequence $z_n = |f^n_\omega (z_0)|$ that stays near $0$ satisfies
	$|z_{n+1} - \mu z_n| \le C z_n^2$ (if $\omega_n = 0$)
	or
	$|z_{n+1} - \frac 1 \mu z_n| \le C z_n^2$ (if $\omega_n = 1$)
	for some $C>0$.
	Now  \eqref{e:Einfty} follows just as in the proof of  \cite[Theorem~5.2]{gh17}.
	
	The remaining statements follow as in the proof of Theorem~\ref{t:inter}.
\end{proof}

\def\cprime{$'$}

\end{document}